\documentclass[12pt]{article}
\usepackage{amssymb}
\usepackage{amsthm}
\newtheorem{theorem}{Theorem}
\newtheorem{proposition}{Proposition}
\newtheorem{corollary}{Corollary}
\newtheorem{lemma}{Lemma}
\newtheorem{definition}{Definition}
\newtheorem{remark}{Remark}
\newtheorem{example}{Example}

\newcommand{\be}{\begin{equation}}
\newcommand{\ee}{\end{equation}}

\newcommand{\R}{\mathbb R}

\newcommand{\E}{\mathbb E}
\newcommand{\Estar}{\mathbb E^*}
\newcommand{\f}{f: X\to\R}
\newcommand{\scalpr}[2]{\langle#1,#2\rangle}
\newcommand{\norm}[1]{\Arrowvert#1\Arrowvert}
\newcommand{\CR}[1]{#1^\uparrow}
\newcommand{\CRsubd}[1]{\partial^\uparrow #1}
\newcommand{\subd}[1]{\partial #1}
\newcommand{\eps}{\varepsilon}
\begin{document}

\title{Characterizations of Pseudolinear and Semistrictly Quasilinear Functions}
\author{ Vsevolod I. Ivanov \\
Department of Mathematics, \\
Technical University of Varna,  9010 Varna, Bulgaria, \\
email: vsevolodivanov@yahoo.com
}



\maketitle

\begin{abstract}

In this paper, we obtain several new  complete characterizations of pseudolinear functions. Two of the results are of first-order and one is derivative free.  The results are derived in terms of the Clarke-Rockafellar subdifferential. Additionally, we prove a characterization of the semistrictly quasilinear functions. It is similar to the derivative free characterization of the pseudolinear functions. We also find the conditions such that a semistrictly quasilinear function become pseudolinear. 


\end{abstract}

\section{Introduction}
\label{intro}
The concepts of pseudolinearity and semistrict quasilinearity provide in a very natural way generalizations of the linearity. These classes  include many useful functions. Among them are, for instance, the linear fractional functions.
Several interesting properties of these of pseudolinear functions appeared in
 \cite{che84,kom93,kor67,rap91}. More generalized convex functions have been studied for example in  the books \cite{cam09,gio04,man69,mar75,mis15}. Quasilinear functions have been studied in \cite{mar75} under the name quasimonotone functions.

In this paper, we obtain more characterizations of pseudolinear functions.  
In Theorem \ref{pl-th3}, we obtain necessary conditions for pseudolinearity, but they are not sufficient. The respective necessary and sufficient conditions are obtained in Theorem 4. Theorem \ref{pl-th1} is well known and we prove it for reader's convenience. Really it is a  generalization to non-differentiable functions of a theorem due to Chew and Choo \cite{che84}. 
 It is interesting which is the largest class of functions such that the conditions of Theorem 3 hold. We prove that they become necessary and sufficient when the function is  semistrictly quasilinear, which implies that the mentioned class is the exact one.
Every pseudolinear function, which is subdifferentiable with respect to the Clarke-Rockafellar subdifferential, is semistrictly quasilinear. We also find the conditions such that a  semistrictly quasilinear function have to satisfy to become pseudolinear.



In the sequel, we suppose that $\E$ is a Banach space. We denote by $\Estar$ its dual and the duality pairing between the vectors $a\in\Estar$ and $b\in\E$ by $\scalpr{a}{b}$, by $\R$ the set of reals, by $\overline\R$ the union $\R\cup\{+\infty\}$, by $B(x,r)$ the closed ball of a center $x$ with a radius $r$. Let $\f$ be a proper extended real-valued function, whose domain is the set 
\[
{\rm dom}(f):=\{x\in \E\mid f(x)<+\infty\}.
\]


\begin{definition}
Let $\f$ be a proper extended real-valued function and $x\in{\rm dom}(f)$. The Clarke-Rocka\-fel\-lar generalized derivative of $f$ at $x$ in direction $v$ is defined by
\[
\CR f(x,v)=\sup_{\eps>0}\,\limsup_{(y,\alpha)\downarrow {}_f x; t\downarrow 0}\,\inf_{u\in B(v,\eps)}\,[f(y+tu)-\alpha]/t,
\]
where $(y,\alpha)\downarrow {}_f x$ means that $y\to x$, $\alpha\to f(x)$, $\alpha\ge f(y)$ (see {\rm \cite{roc80}}), and $y\to x$ implies that the norm $\norm{y-x}$ approaches $0$. If $f$ happens to be lower semicontinuous at $x$ the definition can be expressed in the slightly simpler form
\[
\CR f(x,v)=\sup_{\eps>0}\,\limsup_{y\downarrow {}_f x; t\downarrow 0}\,\inf_{u\in B(v,\eps)}\,[f(y+tu)-f(y)]/t,
\]
where $y\downarrow {}_f x$ means that $y\to x$, $f(y)\to f(x)$. When $f$ is locally Lipschitz, this derivative coincides with the Clarke generalized derivative \cite{cla83}, which is defined by
\[
f^0(x,v)=\limsup_{y\to x; t\downarrow 0}\, [f(y+tv)-f(y)]/t.
\]
The Clarke-Rockafellar subdifferential of $f$ at $x$ is defined as follows:
\[
\CRsubd f(x)=\{x^*\in\E^*\mid\scalpr{x^*}{v}\le\CR f(x,v),\quad\forall v\in\E\}
\]
with the convention that $\CRsubd f(x)=\emptyset$ if $x\notin{\rm dom}(f)$.
The Clarke's subdifferential (or Clarke's generalized gradient) of $f$ at $x$ is defined as follows:
\[
\subd f(x)=\{x^*\in\E^*\mid\scalpr{x^*}{v}\le f^0(x,v),\quad\forall v\in\E\}.
\]

\end{definition}

\begin{definition}
A real-valued function $\f$ is called pseudoconvex (in terms of the Clarke directional derivative) iff the following implication is satisfied
\be\label{2}
f(y)<f(x)\quad\Rightarrow\quad\scalpr{x^*}{y-x}<0,\quad\forall x^*\in\subd f(x).
\ee
\end{definition}

Recall that a  real function is said to be quasiconvex on a convex set $X$ iff, 
\[
f[x+t(y-x)]\le \max\{f(x),f(y)\},\quad\forall  x\in X\, \;\forall y\in X,\;\forall t\in[0,1].
\]
The following result is due to Daniilidis, Hadjisavvas \cite[Proposition 2.2]{dan99}.

\begin{lemma}\label{lema2}
Let $\f$ be a lower semicontinuous  pseudoconvex function with a convex domain. Then $f$ is quasiconvex.
\end{lemma}

The following result is a particular case of Lemma 3 in the paper by Ivanov  \cite{jogo-3}:
\begin{lemma}\label{lema1}
Let $f$ be a  lower semicontinuous  proper extended pseudoconvex  real function, defined on some open convex set in a Banach space $\E$, which contains the convex set $S$. Then the following implication holds
\[
x\in S,\; y\in S,\; f(y)\le f(x)\quad\Rightarrow\quad\scalpr{x^*}{y-x}\le 0,\quad\forall x^*\in\CRsubd f(x).
\]
\end{lemma}

The next result were also derived by Ivanov \cite{jogo-3}:

\begin{lemma}\label{th2} 
Let $\f$ be a lower semicontinuous and radially continuous proper extended real-valued function with a convex domain.
Then $f$ is pseudoconvex if and only if there exists a positive function $p:\E\times\E\times\Estar\to (0,+\infty)$ with

\be\label{3}
\begin{array}{c}
p(x,y,x^*)\;\scalpr{x^*}{y-x}+p(y,x,y^*)\; \scalpr{y^*}{x-y}\le 0, \\
\forall (x,y)\in X \times X,\;\forall (x^*,y^*)\in\subd f(x)\times\subd f(y).
\end{array}
\ee
\end{lemma}

\section{Characterizations of pseudolinear functions}
\label{s7}
In this section, we apply the characterizations of pseudoconvex functions to obtain characterizations of pseudolinear ones.

Recall that a function $f$ is said to be pseudoconcave iff $-f$ is pseudoconvex. A function $f$ is said to be pseudolinear iff $f$ is both pseudoconvex and pseudoconcave. In the characterizations of pseudoconvex functions, we suppose that $f$ is a proper extended real-valued function, which implies that $f(x)>-\infty$ for every $x\in\E$. We want to apply these results to functions such that both $f$ and $-f$ are proper. Therefore $f$ should be a finite function.

The next theorem is well known. Its proof appears in terms of various derivatives and subdifferential in \cite{ans99,che84,lu05,rez12}.  It was proved initially in  \cite{che84}. Our proof is more general, but  similar. We prove it, because the construction in it is used in the  proof of 
 another result.

\begin{theorem}\label{pl-th1} 
Let $f$ be a  lower semicontinuous  proper extended function with a convex domain $S$.  Then 
$f$ is pseudolinear on $S$ with respect to the Clarke-Rockafellar  directional derivative if and only if
there exists a positive function $p:S\times S\times\E^*\to (0,+\infty)$ with
\be\label{pl2}
f(y)-f(x)=p(x,y,x^*)\,\scalpr{x^*}{y-x},\quad\forall\; x\in S\,\;\forall y\in S,\;\forall x^*\in\CRsubd f(x)
\ee
such that $\CRsubd f(x)\ne\emptyset$.
\end{theorem}
\begin{proof}
Let $f$ be pseudolinear. We prove that there exists a function $p$ satisfying (\ref{pl2}). 
Consider the function, defined as follows: 
\be\label{plp}
p(x,y,x^*)=\left \{
\begin{array}{ll}
\frac{f(y)-f(x)}{\scalpr{x^*}{y-x}}, & \quad\textrm{if}\quad\scalpr{x^*}{y-x}\ne 0 \\
1, &\quad\textrm{if}\quad\scalpr{x^*}{y-x}=0.
\end{array}\right. \ee
We prove that it is positive. Let $\scalpr{x^*}{y-x}>0$. It follows from pseudoconvexity that $f(y)\ge f(x)$. Suppose that it is possible that $f(y)=f(x)$. Then by Lemma \ref{lema1} we obtain that $\scalpr{x^*}{y-x}\le 0$, which is a contradiction. Let $\scalpr{x^*}{y-x}<0$. Since $\partial(sf)(x)=s\partial f(x)$ for every $s\in\R$, then $-x^*\in \partial (-f)(x)$. We conclude from here that $f(y)\le f(x)$. By Lemma \ref{lema1} we obtain that the case $f(y)=f(x)$ is impossible. Therefore $p>0$.

We prove that the function $p$ satisfies (\ref{pl2}). It is enough to show that $\scalpr{x^*}{y-x}=0$ implies that $f(y)=f(x)$. Indeed, assume the contrary. If $f(y)<f(x)$, then by pseudoconvexity we obtain that $\scalpr{x^*}{y-x}<0$, a contradiction. 
If $f(y)>f(x)$, then by pseudoconcavity we again get a contradiction.

Let  $x\in S$,   $y\in S$,  $x^*\in\CRsubd f(x)$ and equation (\ref{pl2}) is satisfied. Obviously (\ref{pl2}) implies that $f$ is pseudolinear.
\end{proof}

\begin{theorem}\label{pl-th2} 
Let $f$ be a locally Lipschitz real-valued function, defined on some open convex set in a Banach spase $\E$, which contains the convex set $S$.
Then $f$ is pseudolinear on $S$ if and only if there exists a positive function $p: S\times S\times\Estar\to (0,+\infty)$ with
\be\label{pl3}
\begin{array}{c}
p(x,y,x^*)\;\scalpr{x^*}{y-x}+p(y,x,y^*)\; \scalpr{y^*}{x-y}=0, \\
\forall (x,y)\in S \times S,\;\forall (x^*,y^*)\in\subd f(x)\times\subd f(y).
\end{array}
\ee 
\end{theorem}
\begin{proof}
Let $f$ be pseudolinear. We prove that inequality (\ref{pl3}) holds. Choose arbitrary $x \in S$, $y\in S$. 
It follows from Theorem \ref{pl-th1} that
there exists a function $p: S\times S\times\E^*\to (0,+\infty)$ with
\be\label{pl4}
f(y)-f(x)=p(x,y,x^*)\;\scalpr{x^*}{y-x},\quad\forall x^*\in\subd f(x) 
\ee
and
\be\label{pl5}
f(x)-f(y)=p(y,x,y^*)\;\scalpr{y^*}{x-y},\quad\forall y^*\in\subd f(y).
\ee
If we add  (\ref{pl4}) and (\ref{pl5}), then we obtain (\ref{pl3}).

The converse claim follows from Lemma \ref{th2}.
\end{proof}

Equation (\ref{pl3})  is a type of monotonicity of the subdifferential of a pseudolinear function. More properties of such maps are studied in \cite{bia03}. More types monotone maps have been defined in \cite{kar76,kar90}.

\begin{theorem}\label{pl-th3} Let $S$ be a convex set, included in some open convex set $\Gamma$ in a Banach space $\E$.
Suppose that $f$ is a locally Lipschitz function, which is  pseudolinear on $S$ with respect to the Clarke's derivative.  Then 
for all $x\in S$, $y\in S$, and $\lambda\in[0,1]$ there exists a number $b>0$, which depend on $x$, $y$, $\lambda$
such that  the following conditions are satisfied:
\be\label{pl6}
f[x+\lambda(y-x)]=\lambda\, b\, f(y)+(1-\lambda\, b) f(x),
\ee
\be\label{pl7}
0<b\le 1/\lambda ,\quad\forall\lambda\in(0,1].
\ee
\end{theorem}
\begin{proof}
Choose arbitrary points  $x\in S$, $y\in S$
and a number $\lambda\in (0,1)$.  Denote $z(\lambda)=x+\lambda(y-x)$. We have $\subd f(z(\lambda))\ne\emptyset$. Take arbitrary $\xi\in\subd f(z(\lambda))$.
It follows from Theorem \ref{pl-th1} that there exists a positive function $q: S\times S\times\E^*\to (0,+\infty)$ such that
\be\label{pl8}
q(z(\lambda),x,\xi)[f(x)-f(z(\lambda))]=\scalpr{\xi}{x-z(\lambda)}=\lambda\scalpr{\xi}{x-y}
\ee
and
\be\label{pl9}
q(z(\lambda),y,\xi)[f(y)-f(z(\lambda))]=\scalpr{\xi}{y-z(\lambda)}=(1-\lambda)\scalpr{\xi}{y-x}
\ee
where $q=1/p$. Let us multiply (\ref{pl8}) by $(1-\lambda)$,
(\ref{pl9}) by $\lambda$, and add the obtained inequalities. Then we obtain that (\ref{pl6}) holds where
\be\label{pl11}
b=q(z(\lambda),y,\xi)/[\lambda\, q(z(\lambda),y,\xi)+(1-\lambda)\, q(z(\lambda),x,\xi)].
\ee
It follows from (\ref{pl11}) that $0<\lambda\, b<1$ if $0<\lambda<1$ and
$x\ne y$.

We prove that $b$ does not depend on $\xi$. Equation (\ref{pl6}) implies that if $f(y)\ne f(x)$, then
\[
b=\frac{f[x+\lambda(y-x)]-f(x)}{\lambda [f(y)-f(x)]}.
\]
Suppose that $f(y)=f(x)$. Since the function $f$ is both pseudoconvex and pseudoconcave, then by Lemma \ref{lema2}, $f$
is both quasiconvex and quasiconcave. Therefore, 
\[
f[x+\lambda(y-x)]\le\max\{f(x),f(y)\}=f(x)\quad {\rm for all}\; \lambda\in [0,1],
\]
\[
f[x+\lambda(y-x)]\ge\min\{f(x),f(y)\}=f(x)\quad {\rm for all}\; \lambda\in [0,1].
\]
We conclude from here that $f[x+\lambda(y-x)]=f(x)$ for all $\lambda\in[0,1]$.
Hence, (\ref{pl6}) is satisfied with $b=1$ for every $\lambda\in[0,1]$. It is seen that really $b$ does not depend on $\xi$.
\end{proof}

The next theorem gives us a derivative-free complete characterization of
pseudolinear functions.

\begin{theorem}\label{pl-th4}
Let $S$ be a convex set in a Banach space $\E$. Suppose that  $f$ is a continuously
differentiable function, defined on some open convex set, which contains $S$. Then the following claims are equivalent:

{\rm (a)} $f$ is pseudolinear on $S$;

{\rm (b)} there is a function $b: S\times S\times [0,1]\to (0,+\infty)$ such that for all $x\in S$, $y\in S$
there exists the limit
\be \label{19}
q(x,y)=\lim_{\lambda\downarrow 0}b(x,y,\lambda),
\ee
 $q(x,y)$ is strictly positive, and
for each $\lambda\in[0,1]$  equation {\rm (\ref{pl6})} and inequality {\rm (\ref{pl7})} are satisfied.
\end{theorem}

\begin{proof}
We prove the implication (a) $\Rightarrow$ (b). Let $f$ be pseudolinear on $S$.
It follows from Theorem \ref{pl-th3} that the function defined by (\ref{pl11}) satisfy (\ref{pl6}) and (\ref{pl7}). We prove that there exists the limit $\lim_{\lambda\downarrow 0}b(x,y,\lambda)$, and it is strictly positive.
Take arbitrary points $x$, $y\in S$.  We prove that $\lim_{\lambda\downarrow 0}q(z(\lambda),x)=1$, where $q=1/p$ and the function $p$ is defined by (\ref{plp}). It follows from the explicit construction of the function $p$ in the proof of Theorem \ref{pl-th1} that
\[
q(z(\lambda),x)=\lambda\scalpr{\nabla f(z(\lambda))}{x-y}/[f(x)-f(z(\lambda))]
\]
if $f(x)\ne f(z(\lambda))$, because  $f(x)\ne f(z(\lambda))$ if and only if $\nabla f(z(\lambda))(x-z(\lambda))\ne 0$. If  $f(x)=f(z(\lambda))$, then $q(z(\lambda),x)=1$. On the other hand we have
\[
\lim_{\lambda\downarrow 0}\frac{f(z(\lambda))-f(x)}{\lambda}=
\lim_{\lambda\downarrow 0}\frac{f[x+\lambda(y-x)]-f(x)}{\lambda}=
\scalpr{\nabla f(x)}{y-x}.
\]
Therefore, using that $f$ is continuously differentiable, we obtain that
\[
\lim_{\lambda\downarrow 0}\frac{\lambda\scalpr{\nabla f(z(\lambda))}{x-y}}{f(x)-f(z(\lambda))}=
\lim_{\lambda\downarrow 0}\frac{\scalpr{\nabla f(z(\lambda))}{x-y}}{\scalpr{\nabla f(x)}{x-y}}=1.
\]
We conclude from here that
\be\label{pl-a1}
\lim_{\lambda\downarrow 0}\, q(z(\lambda),x)=1.
\ee

To prove that 
\be\label{pl-bq}
\lim_{\lambda\downarrow 0}b(x,y,\lambda)=q(x,y)>0
\ee
we consider two cases:

First, $f(y)\ne f(x)$. Then $f(y)\ne f(z(\lambda))$ for all sufficiently small $\lambda>0$. It follows from $q=1/p$ and (\ref{plp}) that
\be\label{pl-b1}
q(z(\lambda),y)=(1-\lambda)\scalpr{\nabla f(z(\lambda))}{y-x}/[f(y)-f(z(\lambda))].
\ee
According to the continuous differentiability we obtain that
\be\label{pl-b2}
\lim_{\lambda\downarrow 0}\, q(z(\lambda),y)=\frac{\nabla f(x)(y-x)}{f(y)-f(x)}=q(x,y).
\ee
Then we conclude from (\ref{pl11}),  (\ref{pl-a1}), (\ref{pl-b2})  that  (\ref{pl-bq}) holds.

Second, $f(y)=f(x)$. We have that $\nabla f(x)(y-x)=0$. Therefore $q(x,y)=1$. It this case $f[x+\lambda(y-x)]=f(x)$ for all $\lambda\in[0,1]$, and
$b(x,y,\lambda)=1$ for every $\lambda\in[0,1]$. Therefore, the required equality (\ref{pl-bq}) is satisfied.
It is seen from (\ref{plp}) that for all $x$ and $y$ such that $f(y)\ne f(x)$ we have $q(x,y)=\scalpr{\nabla f(x)}{y-x}/[f(y)-f(x)]>0$.

The converse claim (b) $\Rightarrow$ (a) is trivial.
\end{proof}

\begin{example}
Consider the function $f:\R^2\to\R$ defined by $f(x)=x_2/x_1$, where 
$$S=\{x=(x_1,x_2)\mid x_1>0\}.$$
 The function $f$ is pseudolinear over $S$. This function satisfies (\ref{pl6}) and (\ref{pl7}). We have 
\[
 (f(x+\lambda(y-x))-f(x))/\lambda=(x_1 y_2-x_2 y_1)/(x_1(x_1+\lambda(y_1-x_1)))
 \]
and
\[
b(x,y,\lambda)=(f(x+\lambda(y-x))-f(x))/(\lambda(f(y)-f(x)))=y_1/(x_1+\lambda(y_1-x_1)).
\]
Therefore $0<\lambda\,b\le 1$ for all $x\in S$, $y\in S$,  $\lambda\in(0,1]$.
\end{example}

\section{On semi\-strict\-ly quasi\-linear and pseudolinear functions} 

It is interesting which is the class of functions such that the necessary conditions from Theorem \ref{pl-th3} become both necessary and sufficient. In this section, we show that this property 
can be generalized and it becomes necessary and sufficient when the function is semistrictly quasilinear.

\begin{definition}
A function $f$, defined on a convex set $S$ is called semistrictly
quasiconvex iff for all $x$, $y\in S$, $\lambda\in (0,1)$ the following
implication holds:
\[
f(y)<f(x)\quad\Rightarrow\quad f[x+\lambda(y-x)]<f(x).
\]
If the function $-f$ is semistrictly quasiconvex, then $f$ is called semistrictly quasiconcave.
\end{definition}

\begin{definition}
A function $f$, defined on a convex set, is called semistrictly quasilinear iff it is both semistrictly quasiconvex and semistrictly quasiconcave.
\end{definition}


\begin{proposition}[\cite{kar67}]\label{pr-kar}
Let $\f$ be a radially lower semicontinuous semistrictly quasiconvex function. Then $f$ is quasiconvex.
\end{proposition}

\begin{proposition} Let $S$ be a convex set, included in some open convex set $\Gamma$ in a Banach space $\E$.
Suppose that $f$ is a continuous function, which is  pseudolinear on $S$.  Then 
$f$ is both semistrictly quasiconvex and semistrictly quasiconcave.
\end{proposition}
\begin{proof}
The claim follows directly from the known theorem that every pseudoconvex function is semistrictly quasiconvex.
\end{proof}

\begin{lemma}\label{lema-s}
A function $f$ defined on a Banach space $\E$ is both semistrictly quasiconvex and semistrictly quasiconcave if and only if
the following implication holds
\[
x\in {\rm dom}\, f,\; y\in {\rm dom}\, f,\;f(y)<f(x),\;\lambda\in(0,1)\quad\Rightarrow\quad f(y)<f[x+\lambda(y-x)]<f(x).
\]
\end{lemma}
\begin{proof}
The proof follows immediately from the definitions of semistrict quasiconvexity and semistrict quasiconcavity.
\end{proof}

It is interesting which is the widest class of functions, which satisfy the conditions (\ref{pl6}) and $0<\lambda b(x,y,\lambda)<1$.

\begin{theorem}
Let $f$ be a continuous function defined on some convex set in a Banach space $\E$. Then $f$ is both semistrictly quasiconvex and semistrictly quasiconcave if and only if  for all $x\in S$, $y\in S$ and $\lambda\in(0,1)$ there exists a number $b>0$,
 which depend on $x$, $y$, $\lambda$ such that  $0<\lambda b(x,y,\lambda)<1$ and Condition (\ref{pl6}) is satisfied.
\end{theorem}
\begin{proof}
Let $f$ be both semistrictly quasiconvex and semistrictly quasiconcave. Consider the function $b(x,y,\lambda)$ defined by
\[
b=\frac{f[x+\lambda(y-x)]-f(x)}{\lambda [f(y)-f(x)]}.
\]
 Let $f(y)<f(x)$ and $0<\lambda<1$. We prove that
\be\label{4}
0<{f[x+\lambda(y-x)]-f(x)}/[f(y)-f(x)]<1.
\ee
It follows from the definition of strict quasiconvexity that $f[x+\lambda(y-x)]<f(x)$. Therefore \[
f[x+\lambda(y-x)]-f(x)/[f(y)-f(x)]>0.
\] 
It follows from Lemma \ref{lema-s} that 
\[
f(y)<f[x+\lambda(y-x)]<f(x).
\] 
Hence (\ref{4}) is satisfied

The case $f(y)>f(x)$ is similar. It follows from semistrict quasiconvexity that $f[x+\lambda(y-x)]<f(y)$. Therefore
$f[x+\lambda(y-x)]-f(x)<f(y)-f(x)$, which implies that (\ref{4}) is also satisfied. Therefore, Condition (\ref{pl6}) holds and
$0<\lambda b<1$.

Let $x\in S$, $y\in S$, $f(x)=f(y)$. By Proposition \ref{pr-kar} $f$ is both quasiconvex and quasiconcave. Therefore, 
\[
f[x+\lambda(y-x)]\le\max\{f(x),f(y)\}=f(x)\quad \textrm{for all}\; \lambda\in [0,1],
\]
\[
f[x+\lambda(y-x)]\ge\min\{f(x),f(y)\}=f(x)\quad \textrm{for all}\; \lambda\in [0,1].
\]
We conclude from here that $f[x+\lambda(y-x)]=f(x)$ for all $\lambda\in[0,1]$.
Hence, (\ref{pl6}) is satisfied with $b=1$ for every $\lambda\in(0,1)$ and $0<\lambda b<1$.

Conversely, suppose that  $x\in S$, $y\in S$, $f(y)<f(x)$, $0<\lambda<1$, and $0<\lambda b(x,y,\lambda)<1$. We prove that $f$ is both semistrictly quasiconvex and semistrictly quasiconcave. Since $0<\lambda b(x,y,\lambda)<1$, then we have 
\[
\lambda\,  b\,  f(y)+(1-\lambda\, b)\, f(x)<\lambda\, b\, f(x)+(1-\lambda\, b)\, f(x)=f(x),
\]
and
\[
\lambda\,  b\,  f(y)+(1-\lambda\, b)\, f(x)>\lambda\,  b\,  f(y)+(1-\lambda\, b)\, f(y)=f(y).
\]
It follows from (\ref{pl6}) that $f(y)<f[x+\lambda(y-x)]<f(x)$. By Lemma \ref{lema-s} $f$ is both semistrictly quasiconvex and semistrictly quasiconcave. 
\end{proof}

The following claim is well known \cite{kor67}.
\begin{proposition}\label{pr10}
Let $S$ be an open convex set in a finite-dimensional space $E$, $f$ be a Fr\'echet differentiable  function, defined on  $S$. Then $f$ is pseudolinear on $S$ if and only if the following sets are equal for all $x\in S$:
$\{y\in S:\; \nabla f(x)(y-x)=0\}$ and $\{y\in S:\; f(y)=f(x)\}$.
\end{proposition}

The next theorem is similar, but different from Proposition \ref{pr10}.

\begin{theorem}\label{th4}
Let $S$ be an convex set in a Banach space $E$, $f$ be a Fr\'echet differentiable semistrictly quasilinear function, defined on some open set $\Gamma$, containing $S$. Then $f$ is pseudolinear on $S$ if and only if the following implication holds:
\be\label{5}
x\in S,\; y\in S,\; \nabla f(x)(y-x)=0\quad\Rightarrow\quad f(y)=f(x).
\ee
\end{theorem}
\begin{proof}
Let implication (\ref{5}) be satisfied. We prove that $f$ is pseudolinear. Take arbitrary $x\in S$, $y\in S$ such that $f(y)<f(x)$. By semistrict quasiconvexity we have $f[x+\lambda(y-x)]<f(x)$ for every $\lambda\in(0,1)$. 
Therefore $\nabla f(x)(y-x)\le0$. It follows from (\ref{5}) that the case 
$\nabla f(x)(y-x)=0$ is impossible, because $f(y)<f(x)$. Hence $f$ is pseudoconvex. Using similar arguments we can prove that $f$ is pseudoconcave. Both pseudoconvexity and pseudoconcavity imply that $f$ is pseudolinear.

Suppose that $f$ is pseudolinear. We prove that implication (\ref{5}) holds. Let $x\in S$, $y\in S$, $\nabla f(x)(y-x)=0$, but $f(x)\ne f(y)$. If $f(y)<f(x)$, by pseudoconvexity we have $\nabla f(x)(y-x)<0$, which is a contradiction. If $f(y)>f(x)$,
 by pseudoconcavity we have $\nabla f(x)(y-x)>0$, which is also a contradiction. Therefore (\ref{5}) holds.
\end{proof}

We prove a generalization of this result.

\begin{theorem}
Let $S$ be a convex subset of an open set $\Gamma$ in a Banach space $\E$.
Suppose that $f:\Gamma\to\R$ is a continuous function, which is both semistrictly quasiconvex and semistrictly quasiconcave on $S$ and $\CRsubd f(x)\ne 0$, $\CRsubd (-f)(x)\ne 0$ for all $x\in S$.  Then $f$ is pseudolinear with respect to the Clarke-Rockafellar subdifferential
 if and only if  both implications hold:

\be\label{6}
x\in S,\; y\in S,\;\xi\in\subd f(x),\;\scalpr{\xi}{y-x}=0\quad\Rightarrow\quad f(y)\ge f(x).
\ee
and
\be\label{7}
x\in S,\; y\in S,\;\eta\in\CRsubd (-f)(x),\;\scalpr{\eta}{y-x}=0\quad\Rightarrow\quad f(y)\le f(x).
\ee
\end{theorem} 
\begin{proof}
Let $f$ be pseudolinear. We prove implication (\ref{6}). Take arbitrary $x\in S$, $y\in S$, $\xi\in\subd f(x)$ such that $\scalpr{\xi}{y-x}=0$. If $f(y)<f(x)$, then by pseudoconvexity we have $\scalpr{\xi}{y-x}<0$, which is a contradiction. 
we have $\scalpr{\xi}{y-x}>0$, which is also a contradiction. Therefore $f(y)=f(x)$. Thus, implication (\ref{6}) is proved. 
The proof of implication (\ref{7}) is similar.

Conversely, suppose that implications (\ref{6}) and (\ref{7}) are fulfiled. We prove that $f$ is pseudoconvex. Let $x$ and $y$ be arbitrary points from $S$. We prove that 
\[
\scalpr{\xi}{y-x}>0,\; \xi\in\CRsubd f(x)\quad{\rm implies}\quad f(y)\ge f(x).
\]
 Indeed, it follows from $\scalpr{\xi}{y-x}>0$ that $\CR f(x,y-x)>0$. By the definition of the Clarke-Rockafellar derivative, there exist $\varepsilon>0$ and sequences 
$\{x_i\}_{i=1}^\infty$, $x_i\in\Gamma$, $\{t_i\}_{i=1}^\infty$, $t_i>0$ such that $x_i\to x$, $t_i\downarrow 0$ and
\[
\inf_{u\in B(y-x,\varepsilon)}[f(x_i+t_iu)-f(x_i)]/t_i>0, \quad\forall i.
\]
Taking the number $i$ sufficiently large we ensure that $x_i\in B(x,\varepsilon)$. Therefore, we have $y-x_i \in B(y-x,\eps)$ and
$f[x_i+t_i(y-x_i)]>f(x_i)$. Using that $f$ is lower semicontinuous and semistrictly quasiconvex we conclude from Proposition \ref{pr-kar} that it is quasiconvex on $S$. Therefore,
\[
f(x_i)<f[x_i+t_i(y-x_i)]\le f(y).
\]
Hence, $f(x)\le\liminf_{i\to\infty} f(x_i)\le f(y)$. 
It follows from the converse implication that  
\[
x\in S,\; y\in S,\; f(y)<f(x)\quad\textrm{imply}\quad \scalpr{\xi}{y-x}\le 0,\quad\forall\; \xi\in\CRsubd f(x).
\]
Therefore, according to implication (\ref{6}), we obtain that $f$ is pseudoconvex.

We prove that $-f$ is pseudoconvex. Let $x$ and $y$ are arbitrary points from $S$. We prove that 
\[
\scalpr{\eta}{y-x}>0,\; \eta\in\CRsubd (-f)(x)\quad{\rm imply}\quad f(y)\le f(x).
\]
 Indeed, it follows from $\scalpr{\eta}{y-x}>0$ that $\CR {(-f)}(x,y-x)>0$. By the definition of the Clarke-Rockafellar derivative, there exist $\varepsilon>0$ and sequences 
$\{x_i\}_{i=1}^\infty$, $x_i\in\Gamma$, $\{t_i\}_{i=1}^\infty$, $t_i>0$ such that $x_i\to x$, $t_i\downarrow 0$ and
\[
\inf_{u\in B(y-x,\varepsilon)}[-f(x_i+t_iu)+f(x_i)]/t_i>0, \quad\forall i.
\]
Taking the number $i$ sufficiently large we ensure that $x_i\in B(x,\varepsilon)$. Therefore, we have $y-x_i \in B(y-x,\eps)$ and
$f[x_i+t_i(y-x_i)]<f(x_i)$. Using that $f$ is upper semicontinuous and semistrictly quasiconcave, we conclude from Proposition \ref{pr-kar} that it is quasiconcave. Therefore,
\[
f(x_i)>f[x_i+t_i(y-x_i)]\ge f(y).
\]
Hence, $f(x)\ge\limsup_{i\to\infty} f(x_i)\ge f(y)$. 
It follows from the converse implication that  
\[
x\in S,\; y\in S,\; f(y)>f(x)\quad\textrm{imply}\quad \scalpr{\eta}{y-x}\le 0,\quad\forall\; \eta\in\CRsubd (-f)(x).
\]
Therefore, by implication (\ref{7}), we obtain that $-f$ is pseudoconvex, which implies that the function $f$ is pseudolinear.
 
\end{proof}

\begin{corollary}\label{cor1}
Let $S$ be an open convex subset in a Banach space $\E$. Suppose that $f:S\to\R$ is a locally Lipschitz semistrictly quasilinear on $S$ function.  Then $f$ is pseudolinear with respect to the Clarke generalized directioanal derivative if and only if the following implication holds:
\[
x\in S,\; y\in S,\;\xi\in\partial f(x),\;\scalpr{\xi}{y-x}=0\quad\Rightarrow\quad f(y)=f(x).
\]
and
\end{corollary} 

\begin{remark}
Theorem \ref{th4} is not a consequence of Corollary \ref{cor1}, because the Clarke subdifferential $\partial f(x)$ does not coincides with the gradient $\nabla f(x)$ when the function is Fr\'echet differentiable, but it is not necessarily continuously differentiable.
\end{remark}


\begin{thebibliography}{}

\bibitem{ans99} Q.H. Ansari, S. Schaible, J.C. Yao: $\eta$-pseudolinearity, Riviste di Matematice per le Scienze Economiche e Sociali, v. 22 (1999) 31--39.

\bibitem{baz79} M.S. Bazaraa, C.M. Shetty: Nonlinear Programming - Theory and Algorithms, John Wiley \& Sons, New York (1979).

\bibitem{bec91} C.R. Bector, C. Singh: B-vex functions,  J. Optim. Theory Appl. v. 71, No 2 (1991) 237--253.

\bibitem{bia03} M. Bianchi, N. Hadjisavvas, S. Shaible: On pseudomonotone maps $T$ for which $-T$ is also pseudomonotone, J. Convex Anal. v. 10,  No 1 (2003) 149--168.

\bibitem{cam09} A. Cambini, L. Martein: Generalized Convexity and Optimization, Lecture Notes in Econom. and Math. Systems v. 616, Springer, Berlin (2009).


\bibitem{che84} K.L. Chew, E.U. Choo: Pseudolinearity and efficiency, Math. Program. 28 (1984) 226--239.

\bibitem{cla83} F.H. Clarke:  Optimization and Nonsmooth Analysis, A Wiley-Interscience Publication, John Wiley \& Sons, New York (1983).


\bibitem{dan99} A. Daniilidis, N. Hadjisavvas: On the subdifferentials of quasiconvex and pseudoconvex functions and cyclic monotonicity, J. Math. Anal. Appl. v.  237 (1999) 30--42.

\bibitem{gio04} G. Giorgi, A. Guerraggio, J. Thierfelder: Mathematics of Optimization, Elsevier, Amsterdam (2004).

\bibitem{jogo-3} V.I. Ivanov: Characterizations of pseudoconvex functions and semistrictly quasiconvex ones, J. Glob. Optim. v. 57 (2013) 677--693.

\bibitem{kar67} S. Karamardian: Strictly quasi-convex (concave) functions and duality in mathematical programming,
J. Math. Anal. Appl.  v. 20 (1967) 344--358.

\bibitem{kar76} S. Karamardian: Complementarity over cones with monotone and pseudomonotone maps, J. Optim. Theory Appl. v. 18 (1976) 445--454.

\bibitem{kar90} S. Karamardian, S. Schaible: Seven kinds of monotone maps, J. Optim. Theory Appl. v. 66, No 1 (1990) 37--46. 

\bibitem{kom93} S. Koml\'osi: First and second order characterizations of pseudolinear functions, European J. Oper. Res. v. 67 (1993) 278--286. 

\bibitem{kor67} K.O. Kortanek, J.P. Evans: Pseudoconcave programming and Lagrange regularity, Oper. Res v. 15 (1967) 882--892.

\bibitem{lal07} C.S. Lalitha, M. Mehta: A note of pseudolinearity in terms of bifunctions, Asia Pacific J. Oper. Res. v. 24 (2007) 83--91.

\bibitem{lu05} Q.H. Lu, D.L. Zhu: Some characterizations of locally Lipschitz pseudolinear functions, Math. Appl., v. 18 (2005) 272--278.

\bibitem{man69} O.L. Mangasarian: Nonlinear Programming, Mc Graw-Hill, New York (1969).

\bibitem{mar75} B. Martos: Nonlinear Programming Theory and Methods, Akademiai Kiado, Budapest (1975).

\bibitem{mis15} S.K. Mishra, B.B. Upadhyay : Pseudolinear Functions and Optimization, Chapman \& Hall Book, Boca Raton  (2015).

\bibitem{rap91} T. Rapcs\'ak: On pseudolinear functions, European J. Oper. Res. v. 50 (1991) 353--360.

\bibitem{rez12} M. Rezaci, H. Gazor: Nonsmooth pseudolinearity, Journal of Mathematical Extensions, v. 6, no 4 (2012) 63--77

\bibitem{roc80} R.T. Rockafellar: Generalized directional derivatives and subgradients of nonconvex functions, Canadian Journal of Mathematics, v. 32 (1980)  257--280.

\bibitem{sch82} S. Schaible: Bibliography in fractional programming, Zeitschift fur Operations Research, v. 26, no 1 (1982)  211-241.

\bibitem{sch83} S. Schaible, T. Ibaraki: Fractional programming, European J. Oper. Res, v. 12, no 4 (1983) 325--338  

\bibitem{sta97} I.M. Stancu-Minasian: Fractional Programming. Theory, Methods, Applications, Mathematics and Its Applications, v. 409, Springer, Berlin (1997).



\end{thebibliography}
\end{document}